\documentclass[12pt]{article}
\usepackage{amsmath}
\usepackage{epsfig}                     
\usepackage{amssymb}
\usepackage{fullpage}
\newtheorem{thm}{Theorem}

\newtheorem{defn}{Definition}

\newtheorem{claim}{Claim}

\newtheorem{lemma}[thm]{Lemma}

\newtheorem{construction}[thm]{Construction}

\newtheorem{prop}[thm]{Proposition}
\newtheorem{example}{Example}

\newenvironment{proof} { \emph{Proof.} } { {\rule{2mm}{2mm}}\\ }

\newcommand{\bea}{\begin{eqnarray*}}
\newcommand{\eea}{\end{eqnarray*}}

\newcommand{\R}{\mathbb{R}}
\newcommand{\Z}{\mathbb{Z}}
\newcommand{\Q}{\mathbb{Q}}

\newcommand{\e}{\epsilon}

\newcommand{\bm}{\begin{pmatrix}}
\newcommand{\fm}{\end{pmatrix}}

\title{Envelopes of certain solvable groups}
\author{Tullia Dymarz}

\begin{document}
\maketitle
\begin{abstract} 
A discrete subgroup $\Gamma$ of a locally compact group $H$ is called a \emph{uniform lattice} if  the quotient $H/\Gamma$ is compact. Such an $H$ is called an \emph{envelope} of $\Gamma$. 
In this paper we study the problem of classifying envelopes of various solvable groups including the solvable Baumslag-Solitar groups,  lamplighter groups and certain abelian-by-cyclic groups.
Our techniques are geometric and quasi-isometric in nature.
In particular we show that for every $\Gamma$ we consider there is a finite family of preferred \emph{model spaces} $X$ such that, up to compact groups, $H$ is a cocompact subgroup of $Isom(X)$.
We also answer problem 10.4 in \cite{FM3} for a large class of abelian-by-cyclic groups. 
 \end{abstract}

\section{Introduction}

The problem of classifying envelopes of finitely generated groups was initiated by Furstenberg in \cite{Furs} where he coined the term \emph{envelope} and  proposed classifing which Lie group envelopes can occur. Mostow, Margulis, and  Prasad completed this program in \cite{Mostow,Margulis1,Margulis2,Prasad}. There is also earlier work of Malcev on nilpotent Lie groups \cite{Mal}. 
In \cite{F}, Furman revisited the case of lattices in semisimple Lie groups, but this time for arbitrary second countable envelopes. Recently Furman-Bader-Sauer have announced classification of a wider class of envelopes of groups extending Furman's previous results. 
Also in the same spirit, in \cite{MSW2} Mosher-Sageev-Whyte classified all envelopes of virtually free groups.
In this paper we prove the first results of this kind for classes of solvable groups.

\begin{thm}\label{solbsdlthm} Let $H$ be an envelope of a finitely generated group $\Gamma$.
\begin{enumerate}
\item  If $\Gamma$ is a lattice in the three dimensional solvable Lie group $Sol$ then, up to compact kernel, $H$ embeds as a cocompact subgroup of $Isom(Sol)$. 
\item If $\Gamma=\Gamma_m$ is  
the solvable Baumslag-Solitar group $$BS(1,m)=\left< a,b \mid aba^{-1}=b^m\right>$$ then, up to compact kernel, $H$ embeds as a cocompact subgroup of the isometry group $Isom(X_p)$ where $X_p$ is the standard model space for $\Gamma_p$.
\item If $\Gamma=\Gamma_m$ is a lamplighter group $F \wr \Z$ (with $|F|=m$) then, up to compact kernel, $H$ embeds as a cocompact subgroup of the isometry group $Isom(X_p)$ where $X_p$ is the Diestel-Leader graph $DL(p,p)$ a standard model space for $\Gamma_p$. 
\end{enumerate}
\end{thm}
The description of the model spaces $Sol$ and  $X_m$ can be found in Section \ref{groupsandgeom}.
These three types of groups are treated together to emphasize their similar structure. There are some notable differences however. For lattices in $Sol$ there is only one canonical model space into which $H$ embeds, namely $Sol$ itself, while 
for $BS(1,m)$ and $F \wr \Z$, the model space $X_p$ varies depending on $m$ and $|F|$ and $H$ itself. 
In Section \ref{finmanysec} we show 
that for each group there are only finitely many possibilities for $p$. Note also that a lamplighter group is only solvable as long as the defining finite group $F$ is itself solvable but our proof works for all lamplighters.

\begin{thm}\label{abcthm} Let $\Gamma=\Gamma_M$ be a finitely presented abelian-by-cyclic group given by
$$\Gamma_M=\left< a, b_1, \ldots, b_n \mid ab_ia^{-1}= b_1^{m_{1i}}\cdots b_n^{m_{ni}},\ b_ib_j=b_jb_i\right>$$ 
where $M=(m_{ij})$ is an integral matrix.
Let $\bar{M}$ be the absolute Jordan form of $M$ and let $X_{\bar{M}}$ be a standard model space for $\Gamma_M$. If $M$ has either
\begin{enumerate}
\item all eigenvalues off of the unit circle and $\det{M}= 1$ or
\item all eigenvalues of norm greater than one
\end{enumerate}
then any envelope of $\Gamma_M$ embeds cocompactly in a locally compact group that is isomorphic, up to compact groups, to $Isom(X_{\bar{M}^k})$ for some $k \in \Q$.
\end{thm}
The geometry of $X_{\bar{M}}$ and $X_{\bar{M}^k}$ are described in Section \ref{groupsandgeom}.
Note that if $M \in SL_2(\Z)$ is a hyperbolic matrix then $\Gamma_M$ is a lattice in $Sol$ and if $M=[m]$ is a one by one matrix
then $\Gamma_M=BS(1,m)$ so Theorem \ref{abcthm} covers the first two cases of Theorem \ref{solbsdlthm} but there is a subtle difference in the conclusions. Theorem \ref{solbsdlthm} shows that there is a compact normal subgroup $K$ such that  $H/K \subset Isom(X)$. For Theorem $2$
 by ``up to compact groups'' we mean that 
 $\Gamma$  is contained in $H' \subset Isom(X_{\bar{M}^k})$ 
where $$1 \to H' \to H/K\to K' \to 1$$ with $K,K'$ compact. 
The first case ($\det{M}=1$ and all eigenvalues off the unit circle) should be treated as a generalization of lattices in $Sol$. Indeed in this case we can always take $k=1$ and the model spaces $X_{\bar{M}}$ are all solvable Lie groups. 
The second case should be thought of as a generalization of $BS(1,m)$ and in this case in Section \ref{finmanysec} we show that there are again finitely many choices for $k$ depending on the factoring of $\det{M}$.
These groups and spaces $X_{\bar{M}}$ were first studied in \cite{FM3}.

There is another subclass of abelian-by-cyclic groups where we can prove a partial result. This is the class where $M$ has eigenvalues both of norm greater and less than one and $\det{M}> 1$.
  We state the theorem here but refer the reader to Section \ref{rigabcsec} for definitions. 

\begin{thm}\label{weakthm}
Let $\Gamma_M$ be a finitely presented abelian-by-cyclic group with 
$d=\det{M}>1$ and where $M$ has some eigenvalues of norm greater than one and some of norm less that one (and none of norm one).  Then if $\Gamma_M$ is a cocompact lattice in a locally compact group $H$ then 
$H$ is a uniform subgroup of 
$$Sim_{\bar{M}_1}(\R^{n_1}) \times ASim_{\bar{M}_2}(\R^{n_2} \times \Q_{d^k})$$
where $\bar{M}_1$ is the matrix containing all eigenvalues of $M$ with norm greater than one and $\bar{M}_2$ contains the inverses of all eigenvalues of $M$ of norm less than one.
\end{thm}

In fact this theorem applies to all $\Gamma_M$ with no eigenvalues on the unit circle with the understanding that 
if $\det M=1$ then there is no $\Q_d$ factor and $ASim$ is replaced with $Sim$. 
If all of the eigenvalues are greater than one then $ASim_{\bar{M}_2}(\R^{n_2} \times  \Q_d^k)$ is replaced with $Sim_{\bar{M}_2}(\R^{n_2})$. 
Theorem \ref{weakthm} is an intermediate technical step that is used to prove Theorem \ref{abcthm}.

\section{Outline}

The strategy for analyzing envelopes of the above mentioned groups follows the strategy applied by Furman in \cite{F} for analyzing cocompact envelopes of lattices in rank one symmetric spaces. 

First, if $\Gamma \subset H$ is a cocompact lattice then using a construction of Furman (see Construction \ref{Furmanconstruct} below) we get an embedding, up to compact kernel, of $H$ into the quasi-isometry group $QI(\Gamma)$ as a uniform subgroup. Since $\Gamma$ and any model space $X$ for $\Gamma$ are quasi-isometric Construction \ref{Furmanconstruct} also gives us a uniform embedding of $H$ into $QI(X)$. We use this embedding precisely because $Isom(X)$ also maps naturally to $QI(X)$ as a uniform subgroup. For all of our spaces $Isom(X)$ actually embeds into $QI(X)$ up to compact kernel. (This is not the case, for example, for $X=\R$).  
In order to show that the embedding of $H$ into $QI(X)$ has only compact kernel we rely on a topology on uniform subgroups of $QI(X)$ developed by Whyte in \cite{wh}. Since this paper has not appeared yet we reproduce parts of it in Section \ref{topsec}. 
Furman's construction is explained in Section \ref{embedsection}.

Next, to understand $Isom(X)$ and $QI(X)$ better
we describe how
all of the above groups have model spaces that can be constructed from certain $CAT(-1)$ spaces
using the so called \emph{horocyclic product}.
This process is described in Section \ref{groupsandgeom}. 
By viewing these model geometries as horocyclic products we are able to interpret 
$Isom(X)$ as a product of \emph{similarities} of the visual boundaries of these $CAT(-1)$ spaces.
Then by appealing to various quasi-isometric rigidity theorems we can interpret 
$QI(X)$ as 
a product of \emph{bilipschitz} maps of these boundaries. This viewpoint has already been used in certain cases like lattices in $Sol$, $BS(1,n)$, and $F \wr Z$ but has not been applied to its full extent to abelian-by-cyclic groups (notably for the ones that are not lattices in solvable Lie groups). 

Finally in Section \ref{conjsec} we study the analytic properties of bilipschitz maps on these boundaries. This analysis leads us in Section \ref{rigsec} to conclude that the envelope $H$ must actually lie in the subgroup identified with $Isom(X)\subset QI(X)$. In the case of lattices in $Sol$, $BS(1,n)$, $F \wr Z$ this analysis is straightforward since there are 
existing theorems that show that uniform subgroups of bilipschitz maps of their boundaries can be conjugated to groups acting by similarities.
For general abelian-by-cyclic groups the analysis is more technical since the existing theorems only give partial information. It is this analysis that makes the case of $\det{M}\neq 1$ with eigenvalues of norm greater or less than one intractable and leaves us only with the partial results of Theorem \ref{weakthm}.

\section{Preliminaries}\label{presec}

\subsection{Analysis}
\begin{defn}[Quasi-isometry]\label{qidef} We say that $f: X \to Y$ is a $(K,C)$ \emph{quasi-isometry} if 
$$-C + {K} d(x,y) \leq d(f(x),f(y))\leq K d(x,y)+C$$
and the $C$ neighborhood of $f(X)$ is all of $Y$. 
\end{defn}

\begin{defn} The quasi-isometry group $QI(X)$ is the group of equivalence classes of self quasi-isometries $f:X \to X$ where $f \sim g$ if $d_{sup}(f,g) < \infty$. We write $[f]\in QI(X)$ to denote the equivalence class of $f$. 
We say that a subgroup $U \subset QI(X)$ is \emph{uniform} if there exists fixed $(K,C)$ such that each equivalence class in $U$ has at least one representative that is a $(K,C)$ quasi-isometry. 
\end{defn}

\begin{defn}[Similarity]\label{simdef} We say a map $f: X \to Y$ is a \emph{similarity} with similarity constant $s$ if 
$$ d(f(x),f(y))=s\ d(x,y).$$
\end{defn}

\begin{defn}[Bilipschitz/Quasi-similarity]\label{qidef} We say that $f: X \to Y$ is a \emph{bilipschitz} if 
$$a\ d(x,y) \leq d(f(x),f(y))\leq b\ d(x,y).$$
If we can chose $a=1/K$ and $b=K$ then we say $f$ is a $K$-\emph{bilipschitz} map.
If we can chose $a=s/K$ and $b=sK$ then we say that $f$ is a ($K,s$)-\emph{quasi-similarity}. 
We say that a group of bilipschitz maps/quasi-similarities is \emph{uniform} if $K$ is uniform over all group elements. 
\end{defn}

\noindent{\bf Notation.} We write $Bilip(X)$ to denote the set of all bilipschitz (quasi-similarity) maps of $X$, and $Sim(X)$ to denote all similarities. 

\subsection{Geometry}\label{geomsec}
\begin{defn}[Model space]
We say that a proper geodesic metric space $X$ is a \emph{model space} for a finitely generated group $\Gamma$ if $\Gamma$ acts properly discontinuously and cocompactly on $X$ by isometries. In this case $\Gamma$ is a cocompact lattice in $Isom(X)$.
\end{defn}
Most of the spaces we study in this paper will be constructed from $CAT(-1)$ spaces so we recall here some basic properties of these spaces. For a $CAT(-1)$ space $X$ the usual visual boundary $\partial_\infty X$ can be identified with all geodesic rays starting at a fixed basepoint $x_0 \in X$. We will use instead the \emph{parabolic visual boundary} $\partial_\ell$ which is defined by choosing  $\xi_0 \in \partial_\infty X$ and setting $$\partial_\ell:=\partial_\infty X \setminus \{\xi_0\}.$$ We can treat $\partial_\ell$ as the set of bi-infinite geodesics with one endpoint equal to $\xi_0$.

\begin{defn}[Horosphere]  For a $CAT(-1)$ metric space $X$, fix $x\in X$ and geodesic $\ell$ that is parametrized by length. Define 
$$ h(x)=\lim_{t \to \infty} d(x, \ell(t)) - t.$$
This limit exists by the triangle inequality. The function $h$ is called a horofunction and the level sets of $h$ are called horospheres.
\end{defn}
We  think of $h$ as assigning  a \emph{height} coordinate to points $x\in X$  and so we will sometimes refer to $h:X \to \R$ as a \emph{height function}. We call a geodesic $\ell'(t)$ a \emph{vertical geodesic} if $h(\ell'(t))=t$. In a $CAT(-1)$ space $X$, each $\xi \in \partial_{\infty} X$ is either $\ell^+$ (the point on the boundary define by the ray defined by $\ell$) or there exists a unique vertical geodesic connecting $\xi$ to $\ell^+$. If we take $\xi_0=\ell^+$ then each vertical geodesic with respect to $h$ gives a point on the parabolic visual boundary.

\begin{example} The hyperbolic plane $\mathbb{H}^2$ in the upper half plane model with the coordinates $(x,y)$ has $t=\ln{y}$ as the height coordinate and each each vertical geodesic is given by $$\ell(t)=(x_0, e^t).$$  
\end{example}

\begin{example}\label{treeexample} If $T_{n+1}$ is a regular $n+1$ valent infinite tree and we assign an orientation to edges so that each vertex has $n$ incoming edges and one outgoing edge (pointing ``up")  then, after picking a base point, this orientation induces a height function equivalent to any horofunction given by any coherently oriented geodesic in the tree. (Coherently oriented geodesics are geodesics whose orientation never changes). 
\end{example}

\noindent{}For some of the analysis appearing later we need the following definition.
\begin{defn}[Radial point] Let $U$ be a uniform group of quasi-isometries of $X$.
 We say that $\xi \in \partial_\infty X$ is a \emph{radial point} for $U$ if there exists a sequence of quasi-isometries $g_n:X \to X$ with $[g_n] \in U$ 
 such that for any $x \in X$ and any geodesic ray $\ell$ defining $\xi$ there exist an $R>0$ such that $d_X(g_n(x),\ell)< R$ for all $n$ and  $g_n(x) \to \xi$.
\end{defn}
\section{Topology on $QI(\Gamma)$}\label{topsec}
Our standing assumption will be that $X$ is a path metric space of bounded geometry. (i.e. quasi-isometric to a bounded valence graph). 
The following definition is due to Whyte.

\begin{defn}[QI-tame space] We say that $X$ is a quasi-isometrically tame space if for any 
$(K,C)$  there exists a constant $R_{K,C}$ such that for any quasi-isometry $f:X \to X$ 
 $$ d_{sup}(f,Id) < \infty \Rightarrow d_{sup}(f,Id) < R_{K,C}.$$
\end{defn}

In \cite{wh} Whyte develops a topology of ``coarse convergence" for QI-tame spaces.
We include some of the these details in this paper for completeness.

\subsection{Topology}

\begin{defn}[Coarse convergence] A sequence of $(K,C)$ quasi-isometries $\{f_i\}$ is said to coarsely  converge to $f$ if there exists an $R$ such that 
$$\limsup_{i\to \infty}{d(f_i(x),f(x))} < R$$
in which case we say that $\{f_i\}$ R-coarsely converges to $f$.
\end{defn}
Note that if $X$ is discrete and $f_i$ coarsely converges to $f$ then some subsequence actually converges to a $(K,C)$ quasi-isometry $f'$ with $d_{sup}(f',f) < R$.
\begin{lemma}If $X$ is a QI-tame space then coarse convergence $f_i \to f$  descends to the quasi-isometry group and so we can write $[f_i] \to [f]$.
\end{lemma}
\begin{proof}
If  $\{ f_i \}$ and $\{ f_i' \}$ are two sequence of $(K,C)$ quasi-isometries such that 
$f_i \sim f_i'$ then by the triangle inequality $\{f_i'\}$ coarsely converges to $f$.
$$ \limsup_{i \to \infty} d(f_i'(x),f(x)) \leq \limsup_{i \to \infty} d(f_i'(x),f_i(x))\quad\quad\quad\quad\quad\quad\quad\quad\quad\quad$$$$\quad\quad\quad\quad\quad\quad\quad\quad\quad\quad + \limsup_{i \to \infty} d(f_i(x),f(x))  \leq R_{K,C} + R$$
Likewise if $\{f_i\}$ coarsely converges to $f'$ then 
$$d(f(x),f'(x)) \leq \limsup_{i \to \infty} d(f(x),f_i(x)) + \limsup_{i \to \infty} d(f_i(x),f'(x)) \leq R + R'$$
(and so $d_{sup}(f,f') \leq R_{K,C}$ since $X$ is QI-tame).
\end{proof}\\
In this case we say that $[f_i]$ coarsely converges to $[f]$ and refer only to $(K,C)$ quasi-isometry representatives for some fixed $K$ and $C$.

\begin{lemma} $X$ is QI-tame iff there exists an $R$ such that if a sequence of $(K,C)$ quasi-isometries $f_i$ coarsely converges to $f$ then  $f_i$ $R$-coarsely converges to $f$.
\end{lemma}
\begin{proof}
If $X$ is QI-tame then first discretize $X$.  If $f_i$ coarsely converges to $f$ then the set $\{f_i(x)\}$ is bounded for some (each) $x$. Therefore some subsequence of $f_i$ actually converges to a $(K,C)$ quasi-isometry $f$ which must be within distance $R_{K,C}$ of $f'$ (since $X$ is QI tame). So for each $x$ and $\epsilon$ we have that for large enough $i$ 
$$d(f_i(x),f(x))\leq R_{K,C} + \e.$$
 Conversely, if there exists an $R$ such that every coarsely convergent sequence $R$-coarsely converges then whenever $f' \sim f$ by considering the sequence $f_i=f'$ we see that $d_{sup}(f,f') \leq R$. 
\end{proof}\\
Next, given a uniform subgroup $G \subset QI(X)$ we consider the set $\bar{G}$ of all coarse limits of $G$. We define a topology on $\bar{G}$ by defining a closure operator 
$cl(A)$ that is the set of all coarse limits of nets of elements from $A$. Whyte verifies that $cl(cl(A))=cl(A)$ which is the main property that needs to be checked to ensure that this closure operator does indeed define a topology \cite{Wi}.

\begin{defn}[Closed]\label{closed}  A set $A \subset \bar{G}$ is closed if it contains all of its coarse limits.
\end{defn}

\begin{defn}[Open]\label{open} A set $U \subset \bar{G}$ is open if for any $(K,C)$ quasi-isometry $g$ with $[g] \in U$ and any $s>0$ there exists a finite set $F_{g,s}$ such that if $h$ is a $(K,C)$ quasi-isometry where  $d(h(x),g(x)) < s$ for all $ x \in F_{g,s}$
then $[h]\in U$.
\end{defn}

\begin{prop} If $A$ is closed then $A^C$ is open and if $U$ is open then $U^{C}$ is closed. (And so the two definitions are compatible.)
\end{prop}
\begin{proof}
Let $A$ be closed. Then consider $[h] \in A^{C}$. Suppose that there is some $s$ such that for all finite sets $F$ there exists $h_F$, a $(K,C)$ quasi-isometry such that 
$d(h(x),h_F(x)) \leq s$ but $[h_F] \notin A^C$.
Now let $F_i$ be the ball of radius $i$. Then $h_i:=h_{F_i}$ $s$-coarsely converges to $h$.
But since $A$ is closed we should have $[h]\in A$.

Let $U$ be an open set. 
Now suppose that $[h_i] \in U^{C}$ and that $h_i$ converges coarsely to $h$ with $[h]\in U$. 
Let $s=2R_{K,C}$ and let $F_{h,s}$ be as in the definition. Now since $h_i$  ${R_{K,C}}$-coarsely converges to $h$ let $N$ be large enough so that for $i \geq N$ and for each $x \in F_{h,s}$ 
$$d(h_i(x),h(x))\leq 2R_{K,C}.$$
This shows that for $i\geq N$ we must have $[h_i] \in U$. 
\end{proof}\\
With this topology,  Whyte proves the following in \cite{wh}: 

\begin{prop}[Whyte] For a uniform subgroup $G\subset QI(X)$ of a QI-tame space $X$, the group of all coarse limits $\bar{G}$ is 
a locally compact topological group. 
\end{prop}
\begin{proof}

First we will show that the topology on $\bar{G}$ is Hausdorff. 
For $r>0$ large enough, consider the sets
$$V(x_1, x_2, r):=\{\ [f] \mid \forall f \in [f],\ d(f(x_1),x_2) \geq r \}$$
$$U(x_1,x_2,r):=  cl(V(x_1,x_2,r))^C.$$
Then $U(x_1,x_2,r)$ is open by definition.
 
We will show that if $[f]\neq [g]$ then there exist two disjoint sets $U_1$ and $U_2$ of the form above containing $[f]$ and $[g]$ respectively. If $[f]\neq [g]$ then  for each $N$ there exists an $x_N$ such that 
$d(f(x_N),g(x_N)) \geq N$ (otherwise $f$ would be a bounded distance from $g$).
Now given $r$ pick $x:=x_N$ with $N>2r+R_{K,C}$. Then 
$U_1=U(x,f(x),r)$ and $U_2=U(x,g(x),r)$ are disjoint since if $[h]$ had two representatives $h_1,h_2$  with $d(h_1(x),f(x))\leq r$ and $d(h_2(x),g(x)) \leq r$ then $d(g(x),f(x)) \leq 2r+R_{K,C}$.

Next we note that $U(x,x,r)$ is an open neighborhood of the identity map. If we let 
$A=cl(U(x,x,r))$ then $A$ is a compact neighborhood of the identity since any sequence of quasi-isometries that fixes an $x$ up to a bounded distance has a coarsely convergent subsequence. 
Finally we note that composition with a quasi-isometry preserves open and closed sets. 
This shows that $\bar{G}$ is a locally compact topological group. 
\end{proof}

\section{Embedding uniform envelopes into $QI(\Gamma)$}\label{embedsection}

In this section we show how, given a uniform lattice embedding $\Gamma \subset H$, we can
embed $H$ continuously and with compact kernel as a uniform subgroup of $QI(\Gamma)$ (or equivalently onto a uniform subgroup of $QI(X)$ if $X$ is a model space for $\Gamma$). 
Recall that there is a standard embedding $\rho: \Gamma \to QI(\Gamma)$ given by $\rho(\gamma)=[L_\gamma]$ where $L_\gamma$ is given by left multiplication by $\gamma$.
We use 
 a construction of Furman's from \cite{F} to define a map from $\Phi:H \to QI(\Gamma)$.  Then we show, using our previous work from Section \ref{topsec} on the construction of topologies on uniform subgroups of $QI(X)$,  that $\Phi$ is a continuous map from $H$ onto the uniform subgroup $\Phi(H)$.

\begin{construction}\label{Furmanconstruct} $\Phi: H \to QI(\Gamma)$
\begin{itemize}
\item Let $E \subset H$ be an open neighborhood of the identity $e$, with compact closure such that 
$$H= \bigcup_{\gamma \in \Gamma} \gamma E.$$ 
\item Fix $p: H \to \Gamma$ satisfying $h \in p(h)E$ for each $h \in H$.
\item For each $h\in H$, define  $q_h:\Gamma \to \Gamma$ by the rule $q_h(\gamma):=p(h\gamma)$.
\item Since $e \in E$, we can choose $p$ with $p(\gamma)=\gamma$, so that $q_\gamma(\gamma')=\gamma\gamma'.$
\end{itemize}
\end{construction}
\begin{lemma}[3.3 in \cite{F}]\label{Furman3.3} Let $E,p$ and $\{q_h\}_{h \in H}$ be as above. Then
\begin{description}
\item[](a) Each $q_h : \Gamma \to \Gamma$ is a quasi-isometry of $\gamma$ and its equivalence class $[q_h] \in QI(\Gamma)$ depends only on $h \in H$ (and not on the choice of $E,p$).
\item[](b) The map $\Phi: H \to QI(\Gamma)$, given by $\Phi(h)=[q_h]$, is a homomorphism of (abstract) groups, such $\Phi\vert_\Gamma$
coincides with the standard homomorphism $\rho:\Gamma \to QI(\Gamma).$
\item[](c) $\{q_h\}_{h \in H}$ are $(K,C)$-quasi-isometries for some fixed $K$ and $C$, depending just on $E,p$, and independent of $h \in H$.
\item[](d) There exists a constant $B$ with the following property: given any finite set $F\subset \Gamma$ there is a neighborhood of the identity $V \subset H$ such that $d(q_{h}(\gamma),\gamma)\leq B$ for all $\gamma \in F$ and $h\in V$.

\end{description}
\end{lemma}
Combing Lemma \ref{Furman3.3} with Whyte's coarse convergence topology we get the following proposition. (Compare with Theorem 3.5 in \cite{F}.)
\begin{prop}
Let $\Gamma$ be a $QI$-tame finitely generated group. 
Suppose $\Gamma$ is a cocompact lattice in a locally compact topological group $H$. Then $\Phi: H \to QI(\Gamma)$ (as defined in Construction \ref{Furmanconstruct}) is a continuous homomorphism onto a uniform subgroup of $QI(\Gamma)$ with compact kernel and locally compact, compactly generated image.
\end{prop}
\begin{proof}
Let $G$ be the uniform subgroup of $QI(\Gamma)$ which contains all the coarse limits of $p(H)$. To show that $\Phi$ is continuous, let $U$ be any open subset in $G$ containing the equivalence class of the identity $[id]$. Then by the definition of open set (Definition \ref{open}) we have a finite set $F_{id,s}$ such that if a $(K,C)$ quasi-isometry $q$ maps each element of $F_{id,s}$ at most distance $s$ from itself then $[q] \in U$. 
Let $s=B$ be as in part (d) of Lemma \ref{Furman3.3}. 
Let $V$ be the neighborhood of the idenity in $H$ corresponding to the set $F_{id,B}$. Then $[q_h]$ is in $U$ for all $h\in V$ since by Lemma \ref{Furman3.3} $d(q_h(x),x) \leq B$ for all $x \in F_{id,B}$. 

By QI-tameness, the standard inclusion $p:\Gamma \to QI(\Gamma)$ has finite kernel and $\Phi(\Gamma)$ is discrete in ${\Phi(H)}$. Let $E$ be as in Construction \ref{Furmanconstruct} above. Since $\Phi$ is continuous then $\Phi(\bar{E})$ is compact and therefore $\Phi(H)=\Phi(\Gamma)\Phi(\bar{E})$ is compactly generated and $\Phi(\Gamma)$ is a cocompact lattice in $\Phi(H)$. This also gives us that $\ker{\Phi}$ is compact. 
\end{proof}

\section{The groups and their geometry}\label{groupsandgeom}

The model spaces for our solvable groups that appear in Theorems \ref{solbsdlthm} and \ref{abcthm} are built out of three primary $CAT(-1)$ space. Each of these will be endowed with a preferred height function (horofunction).  We list  them as examples below. See Section \ref{presec} for basic definitions and properties of $CAT(-1)$ spaces. 

\begin{example}  The regular $m+1$ valent tree $T_{m+1}$. (See Example \ref{treeexample} in Section \ref{geomsec}). \end{example}

\begin{example}\label{nchseg} If $M$ is a matrix with all eigenvalues in norm greater than one then the solvable Lie group 
$$G_{M}=\R \ltimes_M \R^n$$ is a negatively curved homogenous space. Here $\R$ acts on $\R^n$ by a one parameter subgroup $M^t \subset GL(n, \R)$.
The height function $h$ is given by $h(t,x)=t$. When $M$ is a scalar matrix the solvable Lie group $G_M$ is hyperbolic space. 
\end{example}
We can combine the above two examples to get what was coined a \emph{millefeuille} space in \cite{amenhyp}. See also \cite{D4} for more details on its geometry. 
\begin{example}[millefeuille space]\label{mfeg} Let $Z_{m,M} \simeq T_{m+1} \times \R^n$  be the fibered product of $T_{m+1}$ and $G_M$. In other words for each oriented line $\ell \in T_{m+1}$ we identify $\ell \times \R^n$ isometrically with $G_M$ in such a way that the height function on $G_M$ coincides with the height function on $T_{m+1}$. The height function on $Z_{M,m}$ is then given by the compatible height functions on $T_{m+1}$ and $G_M$. 
\end{example}

\subsection{Geometric models for solvable groups}  In this section we combine the negatively curved spaces from above to construct model spaces for certain classes of finitely generated solvable groups.

\begin{defn}[Horocyclic product] Given $CAT(-1)$ spaces $X_1, X_2$ with horofunctions $h_{1}, h_{2}$ define their horocyclic product as 
$$ X_1 \times_h X_2=\{ (x_1, x_2) \mid h_{1}(x_1) + h_{2}(x_2)=0\}.$$ 
We give $X_1 \times_h X_2$ the induced path metric from the $L^2$ metric on  $X_1  \times X_2$ rescaled by a factor of $\sqrt{2}$.
\end{defn}

\begin{example} If $X_1, X_2=\mathbb{H}^2$ viewed in the upper half space model with $\ell(t)=(0,e^t)$ then 
$X_1 \times_h X_2$ is the three dimensional $Sol$ geometry. 
\end{example}

\begin{example} If $X_1=T_{n+1},X_2=T_{m+1}$ are $n+1$ and $m+1$ valent trees respectively  then $X_1 \times_h X_2$ are the Diestel-Leader graphs $DL(n,m)$. When $n=m$ these are Cayley graphs for any lamplighter group of the form $F \wr \Z$ with $|F|=n$.
\end{example}

\begin{example} When $X_1= \mathbb{H}^2$ and $X_2=T_{n+1}$ then $X_1 \times_h X_2$ is the complex $X_n$ described in \cite{FM1}. This complex is a model space for the solvable Baumslag-Solitar group $$BS(1,n)=\left< a,b \mid aba^{-1}=b^n \right>.$$
\end{example}

\begin{example}\label{abcexample} Model spaces for the finitely presented abelian-by-cyclic groups 
$$\Gamma_M=\left< a, b_1, \ldots, b_n \mid ab_ia^{-1}=\phi_M(b_i), b_ib_j=b_jb_i\right>$$ 
where $M$ has no eigenvalues on the unit circle
can be written as horocyclic products of $CAT(-1)$ space. 
Suppose the integral matrix $M$ has absolute Jordan form $\bar{M}$ that can be written as 
$$\bar{M}=\bm \bar{M}_1 & 0 \\ 0 & \bar{M}_2^{-1}\fm$$ 
\begin{itemize}
\item If $\det{M}=1$ then a model space for $\Gamma_M$ is  $G_{\bar{M}_1} \times_h G_{\bar{M}_2}$. This space is equivalent to the solvable Lie group $\R \ltimes_{\bar{M}^t} \R^n$ and can be viewed as a generalized version of $Sol$.  
\item If $\det{M}=d>1$ and all eigenvalues are greater than one in norm (i.e $\bar{M}={\bar{M}_1}$) then a model space for $\Gamma_M$ is given by $G_{\bar{M}} \times_h T_{d+1}$. This space  can viewed as a generalization of the spaces that solvable Baumslag-Solitar groups act on. 
\item If $\det{M}=d>1$ and both $M_1$ and $M_2$ are nontrivial then $\Gamma_M$ has 
$G_{\bar{M}_1}\times_h Z_{\bar{M}_2,d}$ as a model space.
\end{itemize}
\end{example}
See Section \ref{rigabcsec} for more details.

\section{Height-respecting quasi-isometries and boundaries}\label{hrsection}
In this section we describe \emph{height-respecting} isometries and quasi-isometries of the various $CAT(-1)$ spaces and the horocyclic products we described in the previous section.  

\begin{defn} A height respecting (quasi-) isometry of a $CAT(-1)$ space or horocyclic product $X$ with height function $h: X \to \R$ is a self (quasi-) isometry that permutes level sets  of $h$ (up to bounded distance) in such a way that the induced map on height is (bounded distance from) a translation.
\end{defn}
It is easy to see that, up to finite index, the isometry groups of the model spaces  we consider consist of only height-respecting isometries.
The key to our analysis is that quasi-isometries of these model spaces are also all height-respecting. This is a highly non-trivial fact which has been proved over a series of papers by many authors. 
In section \ref{qigroupssec} we give detailed references to these results. 
Once we are able restrict to height-respecting quasi-isometries, we can show that quasi-isometries of our model spaces induce height-respecting quasi-isometries of their defining $CAT(-1)$ factors.

Note that in certain cases the $CAT(-1)$ factors have more than just height respecting (quasi-)isometries: for example  $T_{n+1}$ or $\mathbb{H}^n$ or any negatively curved symmetric space have many non height-respecting isometries and quasi-isometries. In other cases, such as $X=G_{\bar{M}}$ when $\bar{M}$ is not a scalar matrix or $Z_{M,m}$, all quasi-isometries are height-respecting (see \cite{X} and \cite{D4}).

\subsection{Parabolic visual boundaries of $CAT(-1)$ spaces}

The interest in height-respecting (quasi-)isometries of $CAT(-1)$ spaces comes from the fact that these kinds of maps induce particularly nice maps on the parabolic visual boundaries of these spaces.
Recall from Section \ref{presec} that the parabolic visual boundary  which we denote $\partial_l X$ can be defined as the space of vertical geodesics with respect to a fixed height function.

\begin{defn}We equip the lower boundary with a \emph{horocyclic visual metric} given by
$$ d_{a,\epsilon}(\xi, \eta)=a^{t_0}$$
where $t_0$ is the smallest height at which the two geodesics $\xi$ and $\eta$ are less than or equal to distance $\epsilon$ apart.
\end{defn}
Here $\epsilon \geq 0$
and for each space $X$ there is an interval of admissible $a$ which makes $d_{a,\epsilon}$ into a metric. In certain cases there are preferred values for $a$ and $\epsilon$ but note that
by changing $\epsilon$ we get bilipschitz equivalent boundary metrics and by changing $a$ we get \emph{snowflake equivalent} boundary metrics.
Note also that by choosing different but quasi-isometric metrics on $X$ we get \emph{quasi-symmetrically equivalent} metrics on $\partial_l X$.  
The above metrics have been used and studied in \cite{D1,DP,X, D4}.  The following proposition can be found for example in \cite{D4} but we will also include a sketch of the proof here. 

\begin{prop}\label{hrprop} Height respecting (quasi-)isometries of $X$ induce (quasi-)similarities of $\partial_l X$ with respect to the horocyclic visual metric $d_{a,\epsilon}$.
\end{prop}
\begin{proof} 
Suppose $\phi:X\to X$ is a quasi-isometry that induces a map on the height factor that is a bounded distance from the translation $t\mapsto t+c$. Suppose $f$ is the induced boundary map. Then for any two boundary points $\xi$ and $\eta$, 
if $x=\xi_{t_0}\in X$ is the first point at which geodesics $\xi_{t}$ and $\eta_{t}$ are distance $\epsilon$ apart and $y=f(\xi)_{t_0}\in X$ is the first point at which $f(\xi)_{t}$ and $f(\eta)_{t}$ are distance $\epsilon$ apart then 
$-C'+ht(x)+c \leq ht(y) \leq C'+ht(x) + c$ so that 
$$a^{-C'}a^{ht(x)}a^c \leq a^{ht(y)} \leq a^{C'}a^{ht(x)}a^c$$
 where $C'$ depends only on the quasi-isometry constants $K,C$. Therefore $f$ is an $(a^c,K')$ bilipschitz map:
 $$ \frac{a^c}{K'}d_{a,\epsilon}(\xi,\eta) \leq d_{a,\epsilon}(f(\xi), f(\eta)) \leq a^cK'd_{a,\epsilon}(\xi, \eta).$$
 If $\phi$ is an isometry then $K'=1$ and so the induced boundary map is a similarity.
\end{proof}
\begin{example} In the case of $T_{m+1}$ the lower boundary can be identified with the $m$-adic numbers $\Q_m$ and if we pick $a=m$ and $\epsilon=0$ then the parabolic visual metric $d_{a,\epsilon}$ coincides with the usual metric on $\Q_m$ (see \cite{FM1}).
\end{example}

\begin{example} In the case of $\mathbb{H}^{n+1}$, the lower boundary with respect to a visual horocyclic metric $d_{e,1}$ is just the usual metric on $\mathbb{\R}^n$. 
\end{example}

\begin{example}\label{metricDM} In the case of $G_{\bar{M}}$ where $\bar{M}$ is diagonal but not a scalar matrix the lower boundary can be identified with $\R^n$ with a visual parabolic metric of the form
$$ D_{\bar{M}}(v,w)=\max\{ |\Delta x_1|^{}, |\Delta x_2|^{\alpha_1/\alpha_2}, \ldots ,|\Delta x_r|^{\alpha_1/\alpha_r}\}$$
where $\alpha_1< \alpha_2 < \ldots <\alpha_r $ are the logarithms of the distinct eigenvalues of $\bar{M}$ and $x_1, \ldots, x_r$ represent vectors in the corresponding eigenspaces.
(The case where $\bar{M}$ is in Jordan form but not diagonalizable is more cumbersome to write down and can be found in \cite{DP}).
We show in \cite{D1} that any bilipschitz map of $(\R^n, D_{\bar{M}})$ has to preserve a flag of foliations defined by the $\alpha_i$'s. Specifically for $v=(x_1, \ldots, x_r) \in \R^n$ we have that 
$$f(v)=(f_1(x_1, \ldots, x_r), f_2(x_2,\ldots,x_r), \ldots,f_r(x_r) )$$ where $f_i$ is bilipschitz in $x_i$ (but only H\"older continuous in the other coordinates).  We denote the set of all such maps by $Bilip_{\bar{M}}(\R^n)$.
We also have that for any $k\in \R$, $\bar{M}$ and $\bar{M}^k$ induce snowflake equivalent metrics $D_{\bar{M}}$ and $D_{\bar{M}^k}$ and so we have 
$$Bilip_{\bar{M}^k}(\R^n)\simeq Bilip_{\bar{M}}(\R^n).$$
\end{example}

\begin{example}\label{metricZh2} If $Z_{\bar{M}, m}$ is the millefeuille space constructed from 
$T_{m+1}$ and $G_{\bar{M}}$ then 
$$\partial_l Z_{\bar{M},m} \simeq \R^n \times \Q_m$$ 
where the metric on $\R^n$ is $D_{\bar{M}}$ as given in Example \ref{metricDM}. It is easy to see that if $f$ is a bilipschitz map of $\partial_\ell  Z_{\bar{M},m}$ then 
$$f(x,y)=(f_1(x,y),f_2(y))$$
where $f_1$ is a $Bilip_{\bar{M}}$ map of $x$ for each $y$ and $f_2 \in Bilip(\Q_m)$. We denote the set of all such maps 
by $Bilip_{\bar{M}}(\R^n \times \Q_m)$ and denote the parabolic visual metric by $D_{\bar{M},m}$. (See \cite{D4} for more details).
\end{example}

\begin{defn}[dilation] We call any map $\delta_t$ in $Bilip_{\bar{M}}(\R^n)$ (or in  $Bilip_{\bar{M}}(\R^n \times \Q_m)$) a $t$-dilation (or just dilation) if $\delta_t$ is a similarity with similarity constant $t$ that is simply multiplication by a constant along each eigenspace of $\bar{M}$. 
\end{defn}
\subsection{Boundaries of horocyclic products}\label{bsec}

For $X_h=X_1 \times_h X_2$ we define two boundaries $\partial_1X_h$ and $\partial_2X_h$ as follows.
Let 
$$\mathcal{G}=\{ \ell \in X_h \}$$
be the set of all vertical geodesics and define two equivalence relations $\sim_1$  and $\sim_2$ by
$$ \ell \sim_1 \ell'  \Leftrightarrow  \lim_{t \to \infty} d(\ell(t), \ell(t')) < \infty $$  
$$ \ell \sim_2 \ell'  \Leftrightarrow  \lim_{-t \to \infty} d(\ell(t), \ell(t')) < \infty. $$  
Then we set
$$ \partial_1 X_h = \mathcal{G}/\sim_1,\ \partial_2 X_h = \mathcal{G}/\sim_2. $$
Note that there are many isometric embeddings of $X_1$ (and $X_2$) into $X_h$.
In particular, if $\pi_i : X_1 \times_h X_2 \to X_i$ are the natural projections then for any $\ell_1$ a vertical geodesic in $X_1$, the inverse image $\pi_1^{-1}(\ell)$ is isometric to $X_2$
and similarly $\pi_2^{-1}(\ell)$ is isometric to $X_1$. 
 Therefore we have that $$\partial_l X_i \subset \partial_i X_h.$$ In fact for any two such embeddings 
$\chi(X_i) \subset X_h$ and $\chi'(X_i) \subset X_h$ we have that $\chi(\ell) \sim_i \chi'(\ell)$ and so we can make the identification
$$\partial_l X_i \simeq \partial_i X_h.$$
\subsection{Isometry groups of horocyclic products}

In this section we describe the isometry groups of horocyclic products in terms of their boundaries.
Consider first  $Isom_{hr}(X_h)$, the group of height respecting isometries of $X_h$. 
(For the spaces we consider this is all of $Isom(X_h)$ up to finite index.)
By Proposition \ref{hrprop} we know that height respecting isometries induce similarities of $\partial_i X_h$ for $i=1,2$ so that we have 
$$Isom_{hr}(X_h) \subset Sim(\partial_l X_1) \times Sim(\partial_l X_2).$$
Additionally, we know that if $f\in Isom_{h}(X_h)$ and the induced height translation is $c$ then the induced similarity boundary maps $f_1$ and $f_2$  have similarity constants $a_1^{c}$ and $a_2^{-c}$
where $a_1, a_2$ depend on the visual metric chosen. The following lemma will be useful later when studying uniform groups of quasi-isometries but we include it here because it relies on the structure of the isometry group.  
 
 \begin{lemma}\label{unifiterate} If $U \subset Sim(\partial_l X_1) \times Sim(\partial_l X_2)$ is induced by a uniform group of height respecting quasi-isometries of $X_h$ then $U \subset Isom_{hr}(X_h). $
 \end{lemma}
 \begin{proof}
 A uniform group of quasi-isometries $U$ has the property that there exists an $R>0$ such that each $f \in U$ induces a map within distance $R$ of a translation on the height factor. 
So for any $f \in U$ we have $f = (f_1, f_2)$ where $f_i$ has similarity constant $a_i^{c_i}$ and $d(c_1,-c_2) \leq 2R$.
If we consider the iterates then $f^s=(f_1^s,f_2^s)$ has similarity constants $n_i^{sc_i}$. If $c_1 \neq -c_2$ then eventually $d(sc_1,-sc_2)>2R$ which is a contradiction.  
  \end{proof}
\subsection{Quasi-isometry groups of horocyclic products}\label{qigroupssec}
Of course the list of groups we describe in Section \ref{groupsandgeom} is not an exhaustive list of groups that arise as horocyclic products but they are chosen precisely because for there groups we know that up to finite index all quasi-isometries are height respecting. This allows us to compute the quasi-isometry groups. The following list gives the quasi-isometry group up to finite index and the reference for where it is proved that quasi-isometries are height-respecting.
\begin{itemize}
\item $QI(Sol)\simeq Bilip(\R) \times Bilip(\R)$  \cite{EFW}
\item $QI(F \wr \Z)\simeq (Bilip(\Q_{|F|}) \times Bilip(\Q_{|F|}))$ \cite{EFW}
\item $QI(BS(1,n))\simeq Bilip(\R) \times Bilip(\Q_n)$ \cite{FM1}

\item$QI(\Gamma_M)\simeq Bilip_{{M_1}}(\R^{n_1}) \times Bilip_{{M_2}}(\R^{n_2})$ 
if $\det{M}=1$ \cite{EFW,P1,P3}

$\quad \quad \quad\ \ \  \simeq Bilip_{{M}}(\R^{n}) \times Bilip(\Q_{d})$ or 

$\quad \quad \quad\ \ \ \ \ \ \ \ \  Bilip_{{M_1}}(\R^{n_1}) \times Bilip_{M_2}(\R^{n_2} \times \Q_{d})$ if $d>1$  \cite{FM3}
\end{itemize}
The last case of $QI(\Gamma_M)$ answers Question 10.4 in \cite{FM3}. Farb and Mosher show that all quasi-isometries are height-respecting but since they do not view their model spaces as horocyclic products of $CAT(-1)$ spaces they are not able to make the identification of the quasi-isometry group with the group of products of bilipschitz maps on both boundaries. 

\subsection{QI-tameness}\label{qitamesec}

\begin{prop} The groups $Sol, F \wr \Z, BS(1,n)$ and $\Gamma_M$ are all QI-tame. 
\end{prop}
The results needed to show QI-tameness of the these groups can be found in the papers where the respective quasi-isometry groups are calculated. 
They amount to showing that a $(K,C)$ quasi-isometry is distance $R_{K,C}$ from a \emph{standard quasi-isometry} and that two different standard quasi-isometries are never a bounded distance apart.

\section{Conjugation Theorems}\label{conjsec}
In this section we recall various theorems that prove that under certain conditions a
uniform subgroup of bilipschitz maps can be conjugated via a bilipschitz map into a group of similarity maps.

\begin{thm}(Theorem 7 in \cite{MSW1})\label{MSWconj} Given $m\geq 2$, suppose that $U \subset Bilip(\mathbb{Q}_m)$ is a uniform subgroup. Suppose in addition that the induced action of $U$ on the space of distinct pairs in $\mathbb{Q}_m$  is cocompact. Then there exists $p \geq 2$ and a bilipschitz homeomorphism $\mathbb{Q}_m \mapsto \mathbb{Q}_p$ which conjugates $H$ into the similarity group $Sim(\mathbb{Q}_p)$. Note that this also means that $p$ and $m$ are powers of a common base (see the appendix in \cite{FM2}).
 \end{thm}

\begin{thm}(Theorem 3.2 in \cite{FM2})\label{FMconj} Let $U \subset Bilip(\R)$ be a uniform subgroup. Then there exists $f\in Bilip(\R)$ that conjugates $U$ into $Sim(\R)$.
\end{thm}
\begin{thm}(Proposition 9  in \cite{D1})\label{Dconj} Let $U \subset Bilip(\R^n)$ for $n\geq 2$ be a uniform separable subgroup that acts cocompactly on distinct pairs of points. Then there exists $f\in Bilip(\R^n)$ that conjugates $U$ into $Sim(\R)$.
\end{thm}
In the case of  general $\partial_l G_{\bar{M}}$ we have a slightly weaker theorem.

\begin{defn}\label{defnalmosttrans1} We write $ASim_{\bar{M}}(\R^n)$ for the set of similarities $Sim_{\bar{M}}(\R^n)$ composed with maps in $Bilip_{\bar{M}}(\R^n)$ of the form   
 $$(x_1,x_2, \cdots , x_r) \mapsto (x_1 + B_1(x_2, \cdots , x_r), x_2 + B_2(x_3,\cdots,x_r), \cdots , x_r + B_r).$$
 We call such maps 
\emph{almost translations}.
\end{defn}

\begin{thm} \label{mytukia2}(Theorem 2 in \cite{D1} and \cite{DP}) Let $U$ be a uniform separable subgroup of $Bilip_{\bar{M}}(\R^n)$ that acts cocompactly on the space of distinct pairs of points of $\R^n$. Then there exists a map $f \in Bilip_{\bar{M}}(\R^n)$ that conjugates $U$ into $ASim_{\bar{M}}(\R^n)$.
\end{thm}
We have a similar theorem for $\partial Z_{\bar{M},m} \simeq \R^n \times \Q_m$.
\begin{defn}\label{defnalmosttrans2} We write $ASim_{\bar{M}}(\R^n\times \Q_m)$ for
the set of similarities $Sim_{\bar{M}}(\R^n \times \Q_m)$ composed with maps in  $Bilip_{\bar{M}}(\R^n \times \Q_m)$ of the form
$$(x_1,x_2, \cdots , x_r,y) \mapsto (x_1 + B_1(x_2, \cdots , x_r,y), \cdots , x_r + B_r(y), y)$$
which we also call \emph{almost translations}.
\end{defn}

\begin{thm}\label{anothertukia}(Theorem 6 in \cite{D4}) Let $U$ be a uniform separable subgroup of $Bilip_{\bar{M}}(\R^n \times \Q_m)$ that acts cocompactly on the space of distinct pairs of points of $\R^n \times \Q_m \simeq \partial_\ell Z_{\bar{M},m}$. Then there exists a map in $Bilip_{\bar{M}}(\R^n \times \Q_m)$ that conjugates
$U$ into $ASim_{\bar{M}}(\R^n\times \Q_p)$ where $p$ is a number such that $p$ and $m$ are powers of a common base. 
\end{thm}

\subsection{Uniform subgroups containing many similarities}

In this section we show how for certain $X$ when a uniform subgroup $U \subset Bilip(X)$ contains ``sufficiently many'' elements of $Sim(X)$ then the conjugating map from the previous theorems (Theorems \ref{FMconj} to \ref{anothertukia}) can be chosen to be the identity map. In other words, 
we don't need to do a conjugation.
We show that this is true for Theorems \ref{FMconj} and \ref{mytukia2} while for Theorem \ref{anothertukia} we show that the conjugating map $f:\R^n \times \Q_m \to \R^n \times \Q_p$ can be chosen to be the identity on the $\R^n$ factor. We also explain why we cannot say anything more about Theorem \ref{MSWconj}.

\subsubsection{On subgroups of $Bilip(\R^n)$}

\begin{defn}
We say that $U \subset Bilip(\R^n)$ contains a \emph{dense group of translations} if $U$ contains translation $T_v(x)=x+v$ for a dense set of $v \in \R^n$. 
\end{defn}

\begin{lemma}\label{noconjR}
Suppose $U \subset Bilip(\R)$  contains a dense group of translations
 then $U \subset Sim(\R)$.
\end{lemma}
\begin{proof} 
This follows directly from \cite{Hink}.
We will show that $\bar{U} \subset Sim(\R)$. Let $Trans(\R)$ be the group of all translations of $\R$. If $U$ contains a dense set of translations then $Trans(\R) \subseteq \bar{U}$. 
By \cite{Hink} Lemma 7 this implies that $\bar{U}=Trans(\R)$ and by \cite{Hink} Lemma 9 we have that $U \subset Sim(\R)$. 
\end{proof}\\
For $n>1$, on top of requiring $U$ to contain a dense group of translations we also require that each $x\in \R^n$ is a radial point for $U$ where the defining sequence can be chosen to be a sequence of similarities (see Section \ref{presec}). To satisfy this condition it is sufficient for $U$ to  contain one sequence of similarities $\{\sigma_i\}$ with similarity constants $\lambda_i \to \infty$ such that the $\sigma_i$ all fix a common $x_0 \in \R^n$ . Then by composing $\sigma_i$ with the appropriate translations we have that every point is a radial point with a sequence of similarities as its defining sequence. 
  
\begin{prop}\label{RNnoconjlem} If $U \subset Bilip(\R^n)$ and each point in $\R^n$ is a radial point for $U$ where the defining sequence can be chosen to be a sequence of similarities $\{\sigma_i\} \subset U$
then $U \subset Sim(\R^n)$.  
\end{prop}
\begin{proof} The key here is that in this situation 
in the proof of Proposition 9  in \cite{D1} (see also \cite{T}) the conjugating map $f$ can be chosen to be affine. Specifically $f$ 
is constructed as a limit $$f=\lim_{i\to \infty} \delta_i a g_i$$
where $g_i \in U, a\in GL_n(\R)$ and $\delta_i \in \R$ is chosen so that $\delta_i g_i$ is $K$-bilipschitz. We will not describe how $a$ is chosen but  $g_i$ can be chosen to be the sequence $\sigma_i$ and $\delta_i$  depends on $\sigma_i$. Specifically we get  $$g_i(x)=\sigma_i(x)=\lambda_i A_i(x-v_i)$$ 
where $A_i \in O(n_i)$ and $ \lambda_i \in \R$ since $\sigma_i$ is a similarity and $v_i \to 0$ since $x$ is a radial point.  We set $\delta_i= \lambda_i^{-1}$.
 Then for some subsequence we get that 
 $$f(x)=\lim_{i\to \infty} aA_i(x-v_i)=\bar{a}x$$
  where
 $\bar{a} \in GL_n(\R)$. Now since $\bar{a} \gamma \bar{a}^{-1}$ is a similarity for each $\gamma \in U$ we must have that 
 $$\gamma(x) =\delta\bar{a}^{-1} A (\bar{a}x + B)$$
 where $\delta \in \R$, $A \in O(n)$ and $B \in \R^n$.
But if $\bar{a}^{-1} A \bar{a}\notin O(n)$ then iterating $\gamma$ contradicts the uniformity of $U$.
Therefore $\gamma$ was already a similarity and no conjugation was needed.
\end{proof}
\subsubsection{On subgroups of $Bilip_{\bar{M}}(\R^n)$}\label{subbilipMsec}

We can adapt Lemmas \ref{noconjR}, \ref{RNnoconjlem} above to arrive with similar results for groups of $Bilip_{\bar{M}}(\R^n)$ maps. The main difference here is that the original conjugation theorem, Theorem \ref{mytukia2} does not necessarily conjugate into $Sim_{\bar{M}}(\R^n)$ but only into $ASim_{\bar{M}}(\R^n)$. We will have to do additional work to show that in fact $U \subset Sim_{\bar{M}}(\R^n)$.  (See Section \ref{rigabcsec} below). Here we will only show that no conjugation is needed to get into $ASim_{\bar{M}}(\R^n)$. 
\begin{prop}\label{noconjM} 
Let $U$ be a uniform subgroup of $Bilip_{\bar{M}}(\R^n)$ where each point in $\R^n$ is a radial point where the defining sequences can always be chosen to be  sequences of similarities with respect to $D_{\bar{M}}$.
Then $U\subset ASim_{\bar{M}}(\R)$.
\end{prop}
\begin{proof}We will only give an outline of the proof. 
Note that having each point be a radial point is equivalent to having $U$ act cocompactly on pairs. Therefore the proof of this lemma can be derived from the proof of Theorem \ref{mytukia2}.
Since in the proof of Theorem \ref{mytukia2} the conjugating map 
is constructed by induction on the number of distinct eigenvalues of $\bar{M}$ where the base case is one of the cases covered in Lemmas \ref{noconjR}, \ref{RNnoconjlem} we only need to consider the induction step. 
But in the induction step either we are in the one dimensional case where we can use a density argument as in Lemma \ref{noconjR}
or we are in the higher dimensional case where the conjugating map is constructed as in Lemma \ref{RNnoconjlem} so again no conjugation was needed. 
\end{proof}

\subsubsection{On uniform subgroups of $Bilip(\Q_m)$ and $Bilip_{\bar{M}}(\R^n \times \Q_m)$}

For uniform subgroups of $Bilip(\Q_m)$ the situation is different from the ones above. 
The main difference is that certain uniform subgroups of $Bilip(\Q_m)$ may only be conjugate into $Sim(\Q_p)$ for $p \neq m$. For example, since $T_4$ is quasi-isometric to $T_2$ via a height-respecting quasi-isometry we can view  $Sim(\Q_4)$ as a uniform subgroup of quasi-similarities of $Bilip(\Q_2)$ but $Sim(\Q_4)$ cannot be conjugate into $Sim(\Q_2)$.
In fact, for each $i,j$, we can view $Sim(\Q_{r^i})$ as a subset of quasi-similarities of $Bilip(\Q_{r^j})$ but only as a subset of similarities if $i=j$. (See Section \ref{finmanysec} for more details).
Therefore we are forced to use Theorem \ref{MSWconj}. 

For $Bilip_{\bar{M}}(\R^m \times \Q_m)$ we have a similar problem for the $\Q_m$ coordinate so we must do the base case conjugation of Theorem \ref{anothertukia}. We also need to redefine what we mean when we say ``a dense set of translations''. 

\begin{defn}\label{densetransdefn} We say that $U\subset Bilip_{\bar{M}}(\R^n \times \Q_m)$ contains a \emph{dense set of translations} if 
for all $(x,y) \in \R^n \times \Q_m$ and $\epsilon > 0$ we have some 
$$\gamma \in U \cap Isom_{\bar{M}}(\R^n \times \Q_m)$$
 with $d(\gamma(x,y),(0,0))<\epsilon$ such that $\gamma(x,y)=(x+v_\gamma, \sigma_\gamma(y))$ where $v_\gamma \in \R^n$ and $\sigma_\gamma \in Isom(\Q_m)$. 
\end{defn}

\begin{prop}\label{noconjMQn}
If $U\subset Bilip_{\bar{M}}(\R^n \times \Q_m)$ contains a dense set of translations and contains sequences of similarities that make each $(x,y)\in \R^n \times \Q_m$ into radial points then there exists a bilipschitz map 
$$f: \R^n \times \Q_m \to \R^n \times \Q_p$$
such that  $f$ conjugates $U$ into $ASim_{\bar{M}}(\R^n \times \Q_p)$ and such that $f$ is the identity on $\R^n$.
\end{prop}

\begin{proof} First we must do the base case conjugation of Theorem \ref{anothertukia}. 
Namely we consider the action of $U$ on $\Q_m$ as given by the restriction of the action of $U$ on $\R^n \times \Q_m$. By Theorem \ref{MSWconj} we can conjugate this restricted action by a map $\sigma:\Q_m \to \Q_p$
to get an action of $U$ by similarities on $\Q_p$.
Then the map  
$$ f(x,y)=(x, \sigma(y))$$ 
is a bilipschitz map from $\R^n \times \Q_m$ to $\R^n \times \Q_p$
that conjugates $U\subset Bilip_{\bar{M}}(\R^n \times \Q_m)$ to a subgroup $U' \subset Bilip_{\bar{M}}(\R^n \times \Q_p)$ where
the $\gamma \in U'$ has the form
$$\gamma(x,y)=(f_1(x_1, \ldots, x_r,y),\ldots,f_r(x_r,y), \sigma_\gamma(y) )$$ 
where $\sigma_\gamma$ is now a similarity of $\Q_p$ and $f_i$ are unchanged by the conjugation.

Note that if $U$ contains a dense subgroup of translations then $U'$ also contains a dense subgroup of translations. 
This is because conjugation by the map $f$ stretch distances at most a fixed amount in the $y$ coordinate and there is no conjugation  on the $\R^n$ factors. Maps of the form $\gamma(x,y)=(x+v_\gamma, \sigma_\gamma(y))$ are sent to maps of the same form where now $\sigma_\gamma(y)$ is a similarity of $\Q_p$. 
Following the same reasoning as in Proposition \ref{noconjM} in combination with Theorem \ref{anothertukia} we see that no further conjugation is needed. In particular $U'$ consists of maps of the form 
$$\gamma(x,y)=\delta_t A(x_1+B_1(x_2, \ldots, x_r,y), \ldots, x_r + B_r(y), \sigma_\gamma(y) )$$ 
where  $\delta_t$ is a dilation, $A \in O(n)$, $(x_1+B_1(x_2, \ldots, x_r,y), \ldots, x_r + B_r(y), y )$ is an almost translation (see Definition \ref{defnalmosttrans2}) and $\sigma_\gamma \in Isom(\Q_p)$.
\end{proof}
\section{Rigidity}\label{rigiditysection}\label{rigsec}
\subsection{Rigidity of lattices in SOL, $BS(1,n)$ and $F\wr Z$.}

In this section we describe all envelopes of lattices in Sol ( $\Gamma_M =\Z \ltimes_M \Z^2$ for $M \in SL_2(\Z)$), the solvable Baumslag-Solitar groups ($BS(1,n)=\left< a,b \mid aba^{-1}=b^n \right>$) and lamplighter groups ($F \wr \Z$ with $|F|=n$). 
These three classes of groups are related in that their model spaces are horocyclic products of combinations of $\mathbb{H}^2$ and $T_{n+1}$ for appropriate $n$. Recall that 
\begin{itemize}
\item the model space for  $\Gamma_M$ is $SOL=\mathbb{H}^2 \times_h \mathbb{H}^2$.
\item the model space for  $BS(1,n)$ is $X_n=\mathbb{H}^2 \times_h T_{n+1}$
\item the model space for  $F \wr \Z$ is $DL(n,n)=T_{n+1} \times_h T_{n+1}$
\end{itemize}

We now prove the three cases of Theorem \ref{solbsdlthm}.\\

\begin{proof}{\it (of Theorem \ref{solbsdlthm}  part 1.)} By Lemma \ref{Furman3.3} we have an embedding with compact kernel and cocompact image
$$ \Phi: H \to  U \subset QI(Sol) \simeq Bilip(\R) \times Bilip(\R).$$
Since the image of $\Gamma_M$ is dense in each factor, we see by Proposition \ref{noconjR} that we must have actually had
$$ U \subset Sim(\R) \times Sim(\R).$$
Finally using uniformity of $U$ and Lemma \ref{unifiterate} we can see that 
$ U \subset Isom(Sol).$
\end{proof}

\begin{proof}{\it (of Theorem \ref{solbsdlthm}  part 2.)}  By a similar argument to the previous theorem we have an embedding with compact kernel and cocompact image
$$ \Phi: H \to U \subset QI(X_n) \simeq Bilip(\R) \times Bilip(\Q_n).$$
Again since the projection of $BS(1,n)$ is  dense in $Bilip(\R)$ we do not need to conjugate in this factor. We do, however, have to use Theorem \ref{MSWconj} to do a conjugation of the $Bilip(\Q_n)$ factor. This gives us
$$ U \subset Sim(\R) \times Sim(\Q_m).$$
Again using Lemma \ref{unifiterate} we conclude that $U \subset Isom(X_m)$.
\end{proof}

\begin{proof}{\it (of Theorem \ref{solbsdlthm}  part 3.)}  In this case we have an embedding with compact kernel and cocompact image
$$ \Phi: H \to U \subset QI(DL(n,n)) \simeq Bilip(\Q_n) \times Bilip(\Q_n) $$
After conjugation we have 
$$ U \subset Sim(\Q_m) \times Sim(\Q_{m'}).$$
Since for $m\neq m'$ we have that $DL(m,m')$ is not quasi-isometric to any finitely generated group \cite{EFW} we must have $m=m'$. 
As before, using Lemma \ref{unifiterate} we can argue that $U \subseteq Isom(DL(m,m))$.
\end{proof}

\subsection{Rigidity of abelian-by-cyclic groups.}\label{rigabcsec}

The goal of this section is to prove Theorems \ref{abcthm} and \ref{weakthm}. First we construct 
construct in more detail the model spaces for the finitely presented abelian-by-cyclic groups 
$$\Gamma_{{M}}=\left< a, b_1, \ldots, b_n \mid ab_ia^{-1}=\phi_M(b_i), b_ib_j=b_jb_i\right>$$ 
where $M=(m_{ij})$ is an integral matrix with $\det{M}\geq1$ and $\phi_M(b_i)=b_1^{m_{i1}} \cdots b_n^{m_{in}}$.
This was first done in \cite{FM3} but our construction and point of view is slightly different. Also, we treat only the cases where $M$ has eigenvalues that are strictly off of the unit circle. Since $\det{M}>1$ we have that $M$ lies on a one parameter subgroup of $GL(n, \R)$. In particular ${M}=e^\mu$ for some $n\times n$ matrix $\mu$ and the one parameter subgroup can be given by ${M}^t=e^{t\mu}$.
Each matrix $\mu$ can be put into real Jordan form
$$S^{-1} \mu S=\mu'= \delta + \nu + \eta$$
 where $\delta$ is diagonal
$\nu$ is superdiagonal and $\eta$ is skew symmetric. 
In this paper we will focus on the cases when $\nu$ is $0$ but the construction is similar for $\nu \neq 0$. 
In this case
$${M}^t=e^{t\mu}=S e^{t\delta} S^{-1} S e^{t\eta} S^{-1}=S\bar{M}^t P^t S^{-1}.$$
where $\bar{M}^t$ is diagonal and $P^t \in O(n)$. We consider the solvable Lie group $G_{\bar{M}} \simeq \R \ltimes_{\bar{M}} \R^n$ defined (as before) by the action of $\bar{M}^t$ on $\R^n$. We endow $G_{\bar{M}}$ with a left invariant metric such that the distance function at each height $t$ is given by 
$$ d_{t,\bar{M}} (x,y)= \| \bar{M}^{-t}(x-y) \|. $$   
If all of the eigenvalues of $\bar{M}$ are greater than one then $G_{\bar{M}}$ is negatively curved (as before). 
If not then we separate $\bar{M}$ into two matrices as before: $\bar{M}_1$ containing the eigenvalues greater than one and $\bar{M}_2$ containing the inverses of eigenvalues less than one (so that $\bar{M}_2$ also has all eigenvalues greater than one). Let $d=\det{{M}} \geq 1$ and let 
$$X_{\bar{M}} \simeq G_{\bar{M}_1} \times_h Z_{\bar{M}_2,d}$$
be the horocyclic product of $G_{\bar{M}_1}$ the negatively curved homogeneous spaces from Example \ref{nchseg}  and $Z_{\bar{M}_2,d}$ the millefeuille space from Example \ref{mfeg}. We have two degenerate cases, if $d=1$ then $Z_{\bar{M}_2,d}=G_{\bar{M}_2}$ or if $\bar{M}=\bar{M_1}$ then $Z_{\bar{M}_2,d}=T_{d+1}$.
As mentioned in Example \ref{abcexample} this horocyclic product is a model space
for the finitely presented abelian-by-cyclic group $\Gamma_{{M}}$. 

The space $X_{\bar{M}}$ is the same one that was defined in \cite{FM3} albeit in a different manner. In \cite{FM3}, $X_{\bar{M}}$ is defined as a fibered product of $G_{\bar{M}}$ and $T_{d+1}$. While the construction in \cite{FM3} obscures the boundary structure it does make it easier to define the action of $\Gamma_M$ on $X_{\bar{M}}$ and so  
to describe this action we use the construction from \cite{FM3}. To do this we first define an action of $\Gamma_M$ on $G_{\bar{M}}$ and then we combine it with the standard action of $\Gamma_M$ on $T_{d+1}$ (it is just the action of $\Gamma_M$ on its Bass-Serre tree $T_{d+1}$). 

Since $M=S\bar{M} P S^{-1}$ where $P \in O(n)$ as above 
then the action of $\Gamma_M$ on $G_{\bar{M}}$ is as follows. 
$$a\cdot (v,t) = (\bar{M}Pv,t+1)$$
$$b_j \cdot (v,t) =  (v + S^{-1} e_j,t)$$
where $e_j$ is the $j$th standard basis vector. These maps are indeed isometries of $G_{\bar{M}}$.
We check that these isometries are compatible with the relations $a b_j a^{-1}=\phi_{M}(b_j)$:
\bea ab_ja^{-1} (v,t) & = & ab_j (P^{-1}\bar{M}^{-1}v, t -1) =  ab_j(S^{-1} M^{-1} Sv, t -1)\\ 
& = & a (S^{-1} M^{-1}S v + S^{-1}e_j, t -1)\\
& = &  (S^{-1} M S[S^{-1} M^{-1}Sv + S^{-1}e_j], t )\\
& = & ( v +  S^{-1} Me_j , t )\\
& = & \phi_M(b_j) \cdot(v,t).
\eea
{\bf Coordinates on $X_{\bar{M}}$.} We can put coordinates $(v,t,y)$ on $X_{\bar{M}}$ where $y\in \Q_d$ and $(t,v)$ are the coordinates on $G_{\bar{M}}=\R \ltimes_{\bar{M}} \R^n$. Note that the coordinate $t$ corresponds to  
negative the height coordinate in $T_{d+1}$.
The action on $X_{\bar{M}}$ is given by
$$a\cdot (v,t,y) = (\bar{M}Pv,t+1, \sigma_a(y)),$$
$$b_j \cdot (v,t,y) =  (v + S^{-1} e_j,t, \sigma_{b_j}(y)).$$
Note that   $\sigma_{a}$ must be a similarity with similarity constant $d$  and  $\sigma_{b_j}\in Isom(\Q_d)$.

We now return to treating $X_{\bar{M}}$ as the horocyclic product $G_{M_1} \times_h Z_{M_2,d}$. Note that the two boundaries of $G_{M_1} \times_h Z_{M_2,d}$ are given by 
 $$\partial_1 X_{\bar{M}}\simeq \R^{n_1},  \quad \partial_2 X_{\bar{M}} \simeq \R^{n_2}\times \Q_d$$
where the metric on $\R^{n_1}$ is given by $D_{\bar{M}_1}$ and the metric on $\R^{n_2}$ is given by  $D_{\bar{M}_2}$.
 (See Section \ref{bsec} for a definitions of  boundaries of a horocyclic product). 
Consider the action of the generators of $\Gamma_M$ on the two boundaries of $X_{\bar{M}}$.  For $g \in \Gamma_{M}$ we write $\gamma_{g,1}$, 
$\gamma_{g,2}$ to denote the maps induced by $g$ on these two boundaries. 

\begin{lemma}\label{denselem}The group $\Gamma_M$ induces a dense set of translations of  $\partial_1 X_{\bar{M}}\simeq \R^{n_1}$ and $\partial_2 X_{\bar{M}} \simeq \R^{n_2}\times \Q_d$ .
\end{lemma}
\begin{proof}
For $\partial_1 X_{\bar{M}}$ we note that the set $$\{ S^{-1}M^te_j  \mid t \in \Z, j=1,\ldots, n \}$$ is dense in $\R^n$. Translations by these are the maps induced on height level sets by conjugates of the group elements $b_j$. Projecting this set to any subspace also gives a dense set. Since the lower boundary can be identified with $E_+$, the span of all eigenvectors with eigenvalues greater than one, we have that $\Gamma_M$ induces a dense set of translations on $\partial_1 X_{\bar{M}}$. If $\det{\bar{M}}=1$ the upper boundary can be identified with $E_-$, the span of all eigenvectors of $\bar{M}$ with eigenvalues less than one, so the same result holds for the upper boundary in this case.

When $\partial_2 X_{\bar{M}}\simeq \R^{n_2} \times \Q_d$ we use the fact that $\Gamma_M$ acts cocompactly on $X_{\bar{M}}$ with fundamental domain a hypercube of height one (i.e. the projection of the fundamental domain to the tree is a single edge containing only one vertex). In particular this means that if we fix the fundamental domain to be at height $t_0$ and pick any other point $s\in X_{\bar{M}}$ at height $t_0$ then there exits $g \in \Gamma_M$ that acts as an isometry on $\partial_2 X_{\bar{M}}$ that sends $s$ into the fundamental domain. On the level of the boundary, this means that if we pick the fundamental domain at a sufficiently large height $t_0$ and so that it contains a vertical geodesic defined by $(0,0) \in \partial_2 X_{\bar{M}}$ then if we let $s\in X_{\bar{M}}$ be a point at height $t_0$ on a geodesic defining $(x,y) \in  \partial_2 X_{\bar{M}}$ then there exist $g\in \Gamma_M$ as above such that
$$d((0,0),\gamma_{g,2}(x,y)) < e^{-t_0} < \epsilon.$$
Since $g\in \Gamma_M$ we have that  $\gamma_{g,2}$ must be of the form $$\gamma_{g,2}(x,y)=(x+v, \sigma(y))$$ as required by 
Definition \ref{densetransdefn}.
\end{proof}

\begin{proof}\emph{(of Theorem \ref{abcthm} and Theorem \ref{weakthm})}
If $\Gamma_M \subset H$ is a cocompact lattice then by Lemma \ref{Furman3.3} we have a cocompact embedding of $\Gamma_M$ up to compact kernel 
$$\Gamma_M \subset H \mapsto U \subset QI(X_{\bar{M}}) \simeq Bilip_{\bar{M}}(\R^{n_1}) \times Bilip_{\bar{M}}(\R^{n_2} \times \Q_m).$$
Combining Theorems \ref{mytukia2} and \ref{anothertukia} we have that after conjugation $U$ acts
on the horocyclic product $$X_{\bar{M}^k}= G_{\bar{M}_1^k} \times_h Z_{\bar{M}_2^k,d^k}$$
 where $k\in \Q$ and $\det(\bar{M}^k)=d^k\in \Z$,
 by 
maps of the form 
$$\gamma(x_1,x_2,y)=[ \delta_{t_1} A_1 (x_{1,1}+B_{1,1}(x_{1,2},\ldots,x_{1,n}), \ldots, x_{1,r} +B_{1,r}),\quad\quad\quad\quad\quad\quad\quad\quad\quad\quad\quad\quad\quad$$
$$\quad\quad\quad\quad\quad\quad\quad\delta_{t_2} A_2(x_{2,1}+B_{2,1}(x_{2,2}, \ldots, x_{2,r},y), \ldots, x_{2,r} + B_{2,r}(y), \sigma_\gamma(y) )].$$ 
where  $\delta_{t_1}, \delta_{t_2}$ are dilations, $A_1 \in O(n_1), A_2 \in O(n_2)$, $(x_{1,1}+B_{1,1}(x_{1,2},\ldots,x_{1,n}), \ldots, x_{1,r} +B_{1,r})$ is an almost translation (as in Definition \ref{defnalmosttrans1}), $(x_{2,1}+B_{2,1}(x_{2,2}, \ldots, x_{2,r},y), \ldots, x_{2,r} + B_{2,r}(y), y )$ is an almost translation (as in Definition \ref{defnalmosttrans2}) and $\sigma_\gamma \in Isom(\Q_{d^k})$.

By iterating $\gamma$ and appealing to uniformity much like in Lemma \ref{unifiterate} we see that $t_1+ t_2=0$. 
This finishes the proof of Theorem \ref{weakthm}. 

We now continue the proof of Theorem \ref{abcthm}. We will prove both cases 1 and 2 concurrently and only give indications where the proofs differ.  The outline of the proof is as follows:
\begin{description}
\item[Step 1:]  We define a homomorphism $\psi:U \to O(n) \times S^1$. 
\item[Step 2:]  We claim that $\ker{\psi}$ consists of elements that act as similarities on both $\partial_1 X_{\bar{M}^k}$ and $\partial_2 X_{\bar{M}^k}$ and therefore by Lemma \ref{unifiterate} we have that $\ker{\psi} \subset Isom(X_{\bar{M}^k})$.
\item[Step 3:] We claim that the image of $\psi$ is compact.
\end{description}
In conclusion we have that up to a compact subgroup (the image of $\psi$) the group $U$ is a subgroup of the isometry group $Isom(X_{\bar{M}^k})$. \\

{\bf Step 1:} First note that we can define a homomorphism 
$$\psi_1':U \to O(n)$$
by $$\psi_1'(\gamma)=A:= [A_1, A_2]$$
where $[A_1, A_2]$ denotes the block matrix with $A_1$ and $A_2$ on the diagonal. 
(If we are in case 2 then $A=A_1$).
Since $\Gamma_M$, hence $U$, are amenable and the only amenable subgroups of $O(n)$ are abelian we have that $\psi_1'(U)$ is contained in a maximal torus. 
This along with the fact that $t_1+ t_2=0$ allows us to define another homomorphism 
$$\psi=\psi_1 \times \psi_2: U \to O(n) \times S^1$$
by $$\psi(\gamma)=( AP^{-t},  e^{2\pi i t})$$
where $t=t_1=-t_2$ and $P\in O(n)$ is given by $M=S\bar{M}PS^{-1}$.
The map $\psi_1$ is actually a homomorphism 
since $P$ is in the image of $\psi_1'$ (specifically it is $\psi_1'(\gamma_a)$ for the generator $a \in \Gamma_M$) and therefore commutes with all $A$ in $\psi_1'(A)$. 
\\

{\bf Step 2:} The kernel of $\psi$ contains $\Gamma_M$ 
and consists precisely of the elements that up to composition with elements of $\Gamma_M$ have boundary maps $\gamma_1, \gamma_2$ of the form
$$\gamma_1(x_1)= (x_{1,1}+B_{1,1}(x_{1,2},\ldots,x_{1,n}), \ldots, x_{1,r} +B_{1,r_1})$$
$$\gamma_2(x_2)=(x_{2,1}+B_{2,1}(x_{2,2}, \ldots, x_{2,r}), \ldots, x_{2,r} + B_{2,r_2}).$$ 
Our goal is to show that $\ker{\psi}$ consists only of maps in $$Sim_{\bar{M}^k_1}(\R^{n_1}) \times Sim_{\bar{M}^k_2}(\R^{n_2})$$
 (or $Sim_{\bar{M}^k}(\R^{n_1}) \times Sim(\Q_{d^k})$ for case 2). In other words, we need all of the $B_{1,i},B_{2,i}$ to be constant maps. We call maps $\gamma$ with all $B_i$ constant
 \emph{straight}. To unclutter the notation we sometimes omit the subscripts $u$ and $1$.
We will show that the projection of the kernel onto each of the two boundaries consists of straight maps and then conclude that then the kernel must also consist of straight maps. 

The projection of the kernel of $\psi$ onto the lower boundary is precisely the elements that, up to composition of powers of $\gamma_{a,1}$, have the form 
$$\gamma(v)= (x_1+B_1(x_2,\ldots,x_n), \ldots, x_r +B_r).$$
The proof that shows that these maps are straight is by induction. We will only give the first step.

\begin{claim}\label{straight1} If  $\gamma(v)= (x_1+B_1(x_2,\ldots,x_n), \ldots, x_r +B_r)$ is the boundary map induced by an element in $\ker{\psi}$ then $B_i(x_{i+1}, \ldots ,x_r)=B_i(0)$ for all $(x_{i+1}, \ldots, x_r)$. 
\end{claim} 
\begin{proof}
Since $\Gamma_M$ induces a dense set of translations (by Lemma \ref{denselem}) then by composing $\gamma$ with appropriate elements of $\Gamma_{{M}}$ we can 
assume that $|B_r|< \epsilon$. We refer to $B_r$ as the additive constant of $\gamma$. 
By iterating $\gamma$ (and composing the iterates with appropriate elements $g_j \in \Gamma_{{M}}$ to keep the additive constant $\epsilon_j$ less that $\epsilon$) we get that
$$\gamma\gamma_{g_n}\cdots\gamma\gamma_{g_1}(v)=(\ldots, x_{r-1}+B_{r-1}(x_r)+ D_1 + B_{r-1}(x_r + \epsilon_1) + D_2 \cdots + B_{r-1}(x_r +\epsilon_{n-1}), x_r+\epsilon_n)$$
where the $D_i$ are constants introduced by the elements $g_j \in \Gamma_M$ and
$$|B_{r-1}(x_r)-B_{r-1}(x_r+\epsilon_i)|<K \epsilon_i^\alpha < K \epsilon^\alpha.$$
Now comparing $\gamma({0})$ with $\gamma(v)$ where $v=(0, \ldots, 0, x_r)$ for some fixed $x_r$ we have 
$$|B_{r-1}(x_r)+ D_1 + B_{r-1}(x_r + \epsilon_1) + D_2 \cdots + B_{r-1}(x_r +\epsilon_{n-1}) \quad \quad \quad \quad \quad \quad \quad \quad \quad \quad$$
$$ -B_{r-1}(0)- D_1 - B_{r-1}( \epsilon_1) -D_2 \cdots - B_{r-1}(\epsilon_{n-1})| \leq K |x_r|^\alpha$$
so that
$$|nB_{r-1}(x_r) - n B_{r-1}(0) | \leq  \sum_{i=1}^{n-1} |B_{r-1}(x_r)-B_{r-1}(x_r+\epsilon_i)|  +  \sum_{i=1}^{n-1}|B_{r-1}(0)-B_{r-1}(\epsilon_i))| 
+ K |x_r|^\alpha
$$ 
and finally for all $n$
$$|B_{r-1}(x_r) -  B_{r-1}(0) | \leq  K\epsilon^\alpha   +  K \epsilon^\alpha  
+ \frac{K}{n} |x_r|^\alpha.
$$ 
\end{proof}\\
When $d=1$ then the upper boundary is treated like the lower boundary. For case 2 the upper boundary maps are already in $Sim(\Q_{d^k})$ for some $k\in \Q$ so there is no further analysis needed.
Since $\Gamma_M \subset \ker{\psi}$ the action of $\ker{\psi}$ on $X_{\bar{M}^k}$ is cocompact and hence Lemma \ref{unifiterate} applies. 
\\

 {\bf Step 3:} We now return to analyze the image of $\psi$.
 \begin{claim} The image of $\psi=\psi_1\times\psi_2$ is compact.
\end{claim} 
\begin{proof} Recall that the image of $\psi_1$ lies in a maximal torus. If $img(\psi)$ is discrete then it is a finite subgroup of $O(n)$ and hence compact. If it is not discrete then for any $A$ in the closure of $img(\psi)$ we have a sequence of maps $\gamma_i \in U$ with rotation constants $A_j \to A$. By composing with maps induced by elements of $\Gamma_M$ we can chose $\gamma_j$ so that 
$$\gamma_j(v)=  (A_j (\delta_{t_j} [x_{1,1}+B_{1,1}^j(x_{1,2},\ldots,x_{1,n}), \ldots, x_{1,r} +B^j_{1,r}],\quad\quad\quad\quad\quad\quad\quad\quad\quad\quad\quad\quad\quad$$
$$\quad\quad\quad\quad\quad\quad\quad\quad\quad \delta_{-t_j}[x_{2,1}+B^j_{2,1}(x_{2,2}, \ldots, x_{2,r}), \ldots, x_{2,r} + B^j_{2,r}]))$$ 
or in case 2
$$\gamma_j(v)=  (A_j (\delta_{t_j} [x_{1,1}+B_{1,1}^j(x_{1,2},\ldots,x_{1,n}), \ldots, x_{1,r} +B^j_{1,r}],  \sigma_{\gamma_j}(y) )$$ 
with $0\leq t_j< 1$ and with $$|B^j_{1,i}(0)|, |B^j_{2,i}(0)|<L \textrm{ and } d(\sigma_{\gamma_j}(0),0)\leq L.$$
These $\gamma_j$ in turn come from a uniform family of quasi-isometries of $X_{\bar{M}^k}$, all of which fix a base point up to bounded amount. (The base point is defined uniquely by picking the origin in $G_{\bar{M}}$ and $y=0\in \Q_p$.)
Therefore some subsequence coarsely converges to a  quasi-isometry that induces the boundary maps $\gamma_1,\gamma_2$. 
We must show that $\gamma=(\gamma_1,\gamma_2)$ has rotation constant $A$. 
Note that if $\gamma$ has rotation constant $A' \neq A$ then
there exists $\beta$ and $v_0$ such that 
$$d(A'v_0,Av_0) \geq \beta \|v_0\|.$$
We can actually pick $v_0$ so that if $v_L$ is any vector with $\|v_L\| \leq L$ 
then for any $s\geq 1$
$$d(A'sv_0,Asv_0) \geq \beta \|sv_0+v_L\| \geq \beta (\|sv_0\|- L).$$
Now since $A_j\to A$ we have that for all $\epsilon$ there exists $J_\epsilon$ such that if $j>J_\epsilon$ then for any $v$
$$d(A_jv, Av)\leq \epsilon \|v\|.$$
Note also that 
$$A_j=(A_{1,j}^1, \ldots, A_{r_1,j}^1,A_{1,j}^2, \ldots, A_{r_2,j}^2)$$
 and so 
 $$A=(A_1^1, \ldots, A_{r_1}^1, A_1^2, \ldots, A_{r_2}^2),\quad 
 A'=({A'}_1^1, \ldots, {A'}_{r_1}^1, {A'}_1^2, \ldots, {A'}_{r_2}^2).$$
We apply the previous facts to each $A_i^1, A_i^2$ in the decomposition. (We drop the superscripts $1,2$ to ease the notation). 
Suppose $A_{i,j} \to A_{i}$ but $A_i \neq A_i'$. Then as before there exists $v_i$ and $\beta$ with $d(A_iv_i,A'_iv_i)\geq \beta\|v_i\|$.
Note that $\gamma_j$ $R$-coarsely converges to $\gamma$ so in particular for any $s\in \R^+$ there is a $J_s$ such that if $j> J_s$ then
$$ d(\delta_{t_j}(A_{i,j}(sv_i + B_{i,j}(0))), \delta_{t}(A'_i(sv_i+B_i(0))))< R.$$ 
Now since $A_{i,j} \to A_i$  then for $j> J_\epsilon$
\bea
d(\delta_{t_j}(A_{i,j}(sv_i + B_{i,j}(0))),\delta_{t_j}(A_{i}(sv_i + B_{i,j}(0))))&\leq& \delta_{t_j} \epsilon \|sv_i +B_{i,j}(0)\| \\
&\leq& N\epsilon (\|sv_i\| + L).
\eea
Combining these two we get 
$$d(\delta_{t_j}(A_{i}(sv_i + B_{i,j}(0))), \delta_{t}(A'_i(sv_i+B_i(0))))<
R+ N\epsilon (\|sv_i\| + L)$$
But $$d(A_i(sv_i+B_{i,j}(0)), A'_i(sv_i+B_i(0))) \geq \beta(\|sv_i\| -L)$$
since $|B_{i,j}(0)|, |B_i(0)|< L$.
So 
$$\frac{1}{N}\beta(\|sv_i\| -L) \leq R + N \epsilon(\|sv_i\|+L).$$
By choosing $s$ large enough and $\epsilon = 1/s$ we see that this is impossible.  
This shows that the image of $\psi_1$ is compact.

If $d>1$ then we are done since after the conjugation the action on the tree factor is by isometries which ensures that only integral powers of the stretch factors can appear. 

If $d=1$ we use a similar argument to the one above to show that if the image of $\psi_2$ is not discrete then it is all of $S^1$. For any $0\leq t_0 < 1$ pick $\gamma_j \in U$ with dilation $\delta_{t_j}$ where $t_j \to t_0$. Up to composition with maps induced by $\Gamma_M$ we can again assume $|B^j_i(0)|< \epsilon$ so that the quasi-isometries inducing $\gamma_j$ all fix the base point (this time is it just $0\in G_{\bar{M}}$) up to a bounded amount. Again using coarse convergence of quasi-isometries we have that some subsequence of these quasi-isometries coarsely converge to a quasi-isometry whose boundary map we label $\gamma$. By looking at  $x_r \neq 0$ large enough we can conclude that $\gamma$ has dilation constant $\delta_{t_0}$.  \end{proof}\\
This claim finishes the proof of the theorem.
\end{proof}

\noindent{\bf Remark.} 
The main reason this approach does not work for the mixed eigenvalues case is that $Isom(\Q_m)$ does not split as a semi-direct product of $Stab(0)$ and a `translation' subgroup. This prevents the arguments from Claim 1 in the proof from going through.   
\\

\section{Finitely many model spaces}\label{finmanysec}

In this section we show that $\Gamma_M$ can only act on finitely many of the model spaces 
(and similarly for $F \wr \Z$).
This also implies that up to compact groups there is only a finite list of locally compact groups in which $\Gamma_M$ (and $F \wr \Z$) can be a lattice. 

First notice that $\Gamma_{M^k}$ is a lattice in $Isom(X_{\bar{M}^i})$ for all $i \leq k$ since $\Gamma_{M^k}$ can be realized as an index $k-i +1$ subgroup of $\Gamma_{M}$. 
Similarly, the lamplighter group $\Gamma_d=F\wr \Z$ has a finite index subgroup isomorphic to $\Gamma_{d^k}=(\oplus_{i=1}^kF) \wr \Z$. (Note however that not all lamplighters $G \wr \Z$ with $|G|=d^k$ are index $k$ in $\Gamma_n$). We will assume without loss of generality that $M$ cannot be written as a power of another matrix and that $d$ is not a proper power. 

Now we are left to show that $\Gamma_{M^i}$ (or $\Gamma_{d^i}$ ) cannot act cocompactly and properly discontinuously on $X_{\bar{M}^k}$ (or $DL(d^k, d^k)$) for $k < i$.
To that end we note that if $\Gamma_{M^i}$ (or $\Gamma_{d^i}$) acts cocompactly on $X_{\bar{M}^k}$ (or $DL(d^k, d^k)$) then it acts cocompactly on $T_{d^k+1}$. Furthermore, this action fixes a point at infinity and preserves the orientation of the tree. (In the lamplighter case this is true after passing to an index two subgroup). This implies that the quotient graph of groups decomposition is an (oriented) cycle of edges where each edge group includes isomorphically into to the initial vertex group and as an index $k$ subgroup into the terminal vertex group. 
By collapsing edges we can assume without loss of generality that the cycle has only one edge and one vertex. (After collapsing we get an action of the group on $T_{d^{ks}+1}$ where $s$ is the number of edges collapses; i.e. on a higher valence tree.) 
 
\begin{prop}
If $G$ acts transitively on $T_{d+1}$ and on $T_{e+1}$ by orientation preserving isometries 
such that there exists a quasi-conjugacy sending one action to the other then $d=e$. 
\end{prop}
\begin{proof}
Recall that a quasi-conjugacy is a quasi-isometry $$\varphi: T_{d+1} \to T_{e+1}$$
such that $\varphi(g\cdot x)$ is bounded distance from $g\cdot \varphi(x)$. 
In particular this implies that vertex stabilizers of one action are contained in vertex stabilizers of the other action. 
 A transitive action on $T_{d+1}$ (resp. $T_{e+1}$) implies that the quotient graph is the graph with exactly one edge and one vertex and where one of the edge to vertex group inclusions is as an index $d$ (resp. index $e$) subgroup. 
If $t$ is the generator in $G$ corresponding to the loop in the graph of groups decomposition of its action on $T_{d+1}$  then the subgroup generated by $t$ must act transitively on height level sets of $T_{e+1}$ (this follows from quasi-conjugacy of the actions). If we let $s\in G$ be the generator of the loop in the graphs of groups decomposition of its action on $T_{e+1}$ then 
$s=tp$ where $p$ is in the kernel of the map to $\mathbb{Z}$ (i.e. the height map).
Now suppose $H=stab(v)$ for $v \in T_{d+1}$. Then $tHt^{-1}$ is an index $d$ subgroup of $H$. 
Furthermore $$sHs^{-1}=tpHp^{-1}t^{-1}=t\ stab(p\cdot v)t^{-1}.$$ Since all of the vertex stabilizers are isomorphic we have that $|H:sHs^{-1}|=d$ as well. 

Let $v'\neq v$ be such that $stab(v) \subset stab(v')$ 
and $|stab(v'): stab(v)|$ is smallest over all possible $v'$.  Label $stab(v')=H'$.
Then 
$$ s^{-1} H' s \subset H \subseteq H'  \subset sHs^{-1} \subseteq sH's^{-1}.$$
Note that $|H': s^{-1}H's|=e$. 
If we let  $|H: s^{-1} H' s|=g$  and $|H':H|=f$ 
then $e=gf$. But we also have that $d=gf$ since $|sHs^{-1}: H'|=g$
and 
$$d=|sHs^{-1}:H|=|sHs^{-1}: H'||H':H|=fg.$$
\end{proof}
\subsection*{Acknowledgements}

I would like to thank Benson Farb for suggesting this problem and Kevin Whyte and Yves de Cornulier for useful conversations. I would also like to thank the referee for helpful comments. 

\bibliographystyle{amsalpha}

\bibliography{TameLC}

\end{document}